\documentclass[10pt,leqno]{amsart}

\usepackage{graphicx}
\usepackage{indentfirst,csquotes}
\usepackage{amssymb,amsthm,amsmath}
\usepackage{xcolor,paralist,hyperref,titlesec,fancyhdr,etoolbox}
\usepackage{lipsum}

\newtheorem{theorem}{Theorem}[]

\newtheorem{example}[theorem]{Example}
\newtheorem{lemma}[theorem]{Lemma}
\newtheorem{proposition}[theorem]{Proposition}
\newtheorem{corollary}[theorem]{Corollary}
\newtheorem{remark}[theorem]{Remark}

\titleformat{\section}
{\normalfont\Large\bfseries\centering}
{\thesection.}
{0.5em}
{}

\titlespacing*{\section}{0pt}{2ex}{1ex}

\hypersetup{
    colorlinks=true,
    linkcolor=black,
    filecolor=black,
    urlcolor=black
}

\begin{document}


\title{On the Cyclicity of Algebraic Lattices}


\author{Maria Fernanda Zordan Bonini}
\address{São Paulo State University (UNESP), 2265 Cristovao Colombo Street, 15054‐000, São Jose do Rio Preto - SP, Brazil}
\email{maria.bonini@unesp.br}

\author{Robson Ricardo de Araujo}
\address{Federal Institute of Sao Paulo (IFSP), 239, Pastor Jose Dutra de Moraes Street, 15808-305, Catanduva - SP, Brazil}
\email{robson.ricardo@ifsp.edu.br}

\author{Antonio Aparecido de Andrade}
\address{São Paulo State University (UNESP), 2265 Cristovao Colombo Street, 15054‐000, São Jose do Rio Preto - SP, Brazil}
\email{antonio.andrade@unesp.br}

\author{Jéfferson Luiz Rocha Bastos}
\address{São Paulo State University (UNESP), 2265 Cristovao Colombo Street, 15054‐000, São Jose do Rio Preto - SP, Brazil}
\email{jefferson.bastos@unesp.br}

%



\begin{abstract}
This work presents theoretical advances in the study of cyclic and quasi-cyclic lattices. First, we provide elementary facts regarding cyclic and quasi-cyclic lattices, and discuss about the cyclicity of some notable lattices. The main contributions of the paper is in the algebraic setting: we investigate cyclic lattices arising from $\mathbb{Z}$-modules in Galois number fields via the Minkowski embedding. We establish necessary and sufficient conditions for an algebraic lattice to be cyclic over both cyclic and general Galois number fields, expressed in terms of naturally associated groups. Moreover, we derive a necessary and sufficient condition for ideal lattices to be cyclic, depending on the factorization of the ideal in the underlying number field.
\end{abstract}

\maketitle

\bigskip

\section{Introduction}

A lattice \( \Lambda \subset \mathbb{R}^n \) is a $\mathbb{Z}$-module generated by $k \leq n$ linearly independent vetors in $\mathbb{R}^n$. In other words, a lattice $\Lambda \subset \mathbb{R}^n$ of dimension $k$ can be written as $\Lambda = B\mathbb{Z}^k$, where $1 \leq k \leq n$ and $B$ is an $n \times k$ matrix of rank $k$ with entries in $\mathbb{R}$. A lattice $\Lambda$ has full rank if $k=n$. The dual lattice of $\Lambda \subset \mathbb{R}^n$ is defined by \(\Lambda^* = \{ \vec{y} \in \mathbb{R}^n \mid \langle \vec{x}, \vec{y} \rangle \in \mathbb{Z} \text{ for all } \vec{x} \in \Lambda \} \). 

Since the end of the last century, lattices have been widely used in different areas of telecommunications and computing, especially in sphere packing and covering problems, in signal transmission in certain types of communication channels, and in the development of post-quantum cryptography techniques \cite{Bernstein2009}. In the context of sphere packing,
full-diversity rotated versions of the lattices $D_n$ and their packing
densities were studied in \cite{AraujoJorge2020}. A relevant approach to constructing lattices involves the use of tools from Algebraic Number Theory. The lattices obtained in this way, called algebraic lattices, are defined as the images of $\mathbb{Z}$-modules in number fields under a homomorphism into a Euclidean space. In particular, algebraic lattices constructed from Abelian extensions have been investigated in connection with error-correcting codes \cite{InterlandoEtAl2021}.

A lattice $\Lambda \subset \mathbb{R}^n$ (of any rank) is called \emph{cyclic} if it is closed under the shift operator $\rho(c_1, \dots, c_n) = (c_n, c_1, \dots, c_{n-1})$, i.e., $\rho(\Lambda) = \Lambda$. The condition that a lattice $\Lambda$ of rank $n$ is cyclic is equivalent to requiring that its automorphism group $\text{Aut}(\Lambda)$ contains the permutation matrix corresponding to the standard $n$-cycle $(1 \, \cdots \, n)$. This condition is nontrivial, and lattices with large automorphism groups are of particular interest in lattice theory. A close relationship exists with cyclic codes via Construction A, which are widely used in digital communications. Cyclic lattices were proposed for application in post-quantum cryptography by Micciancio \cite{Micciancio2007} and later explored in \cite{Lyubashevsky2006,Peikert2006}; such lattices also appear in the NTRU cryptosystem and other post‑quantum schemes \cite{Bernstein2009}. Integer cyclic lattices (those inside $\mathbb{Z}^n$) are especially important in cryptography; for instance, on a broad class of well‑rounded cyclic lattices the Shortest Vector Problem (SVP) equals the Shortest Independent Vectors Problem (SIVP) \cite{lenny2}. 

If the lattice is cyclic and there exists a vector $\vec{a} \in \Lambda$ such that \( \Lambda = \langle \vec{a}, \rho(\vec{a}), \dots, \rho^{n-1}(\vec{a}) \rangle_{\mathbb{Z}}, \)
where $\rho$ denotes the shift rotation linear operator, then $\Lambda$ is said to be  \textit{strongly cyclic} (or  \textit{simple cyclic}). In other words, a strongly cyclic lattice has a cyclic basis. Furthermore, a lattice \( \Lambda \subset \mathbb{R}^n \) is called \textit{quasi-cyclic} if there exists an integer \( k \in \mathbb{N} \) (where \( 2 \leq k \leq n-1 \)) such that \( \rho^k(\Lambda) = \Lambda \), with \( \rho \) denoting the shift rotation operator. In this case, we say \( \Lambda \) is  \textit{\(k\)-cyclic}. This definition generalizes the notion of cyclic lattices, which correspond to the case \( k = 1 \). The trivial case \( k = n \) reduces to the identity transformation; thus, only values \( 1 \leq k \leq n-1 \) are nondegenerate. The \textit{cyclicity factor} $k \in \mathbb{N}$ is defined to be the smallest number $k$ of rotations required to satisfy $\rho^k(\vec{x}) \in \Lambda$ for every $\vec{x} \in \Lambda$.

This work contributes to and refines the theory of cyclic and quasi‑cyclic lattices. We first strengthen the foundations by providing equivalent characterizations that yield a finite algorithm to test cyclicity and quasi‑cyclicity using only basis vectors. 

\begin{theorem}
Let $\Lambda \subset \mathbb{R}^n$ be a lattice with basis $B=\{\vec{v}_1, \ldots, \vec{v}_m\}$, where $m \leq n$. Then $\Lambda$ is cyclic if and only if $\rho(\vec{v}_i) \in \Lambda$ for every $\vec{v}_i \in B$.     
\end{theorem}

This theorem allows us to verify if famous lattices are cyclic by only testing one of their bases. In Section 2, we prove Theorem 1.1 and prove that the cyclicity factor always divides the lattice dimension.

The main contributions of this work consist of studying the cyclicity of algebraic lattices over Galois number fields (Section 3). 
First, we establish a necessary and sufficient condition for an algebraic lattice to be cyclic over cyclic number fields. In this direction, we prove the following theorem:

\begin{theorem}
Let $\mathbb{K}/\mathbb{Q}$ be a cyclic extension of degree $n$ with Galois group $G=\textnormal{Gal}(\mathbb{K}/\mathbb{Q})$. Consider $M$ a $\mathbb{Z}$-submodule of $\mathbb{K}$ with rank $n$. Then $\Lambda_M = \sigma(M)$ is cyclic if and only if $M$ is $G$-stable.
\end{theorem}

Let $\mathbb{K}$ be a Galois number field. By defining a set by restricting the automorphisms of the Galois group to a $\mathbb{Z}$-module $M$, we obtain a new set $G_M = \{\tau|_M~:~\tau\in G\}$. This set is not a subgroup of the Galois group, but we stabilize a condition that guarantees that $G_M$ is a group and a convergence with the Galois group. With this, we prove a generalization of Theorem 1.1, showing for general Galois number fields a criterion to determine whether an algebraic lattice is cyclic, i.e., proving that the algebraic lattice $\sigma(M)$ is cyclic by using the group $G_M$. This is the content of the following theorem, which will be demonstrated in Section 3:

\begin{theorem}\label{teorema_main} Let $\mathbb{K}$ be a Galois number field, $M$ be a $\mathbb{Z}$-module of $\mathbb{K}$, $G$ be the Galois group of $\mathbb{K}/\mathbb{Q}$, and $G_M = \{\tau|_M~:~\tau\in G\}$. The lattice $\Lambda=\sigma(M)$ is cyclic if and only if $G_M$ is a cyclic group.
\end{theorem}

Finally, making a further connection with Hilbert's ramification theory, we completely characterize the (fractional or integral) ideals of $\mathbb{K}$ whose Minkowski embeddings give rise to cyclic lattices by proving the following result:
\begin{theorem}
    Let $\mathbb{K}$ be a cyclic number field, and let $\mathcal{I}$ be a (fractional or integral) ideal in $\mathbb{K}$. The lattice $\Lambda=\sigma(\mathcal{I})$ is cyclic if and only if
$$\mathcal{P}~\textnormal{and}~\mathcal{Q}~\textnormal{are conjugate} ~\Longrightarrow~ \nu_\mathcal{I}(\mathcal{P})=\nu_\mathcal{I}(\mathcal{Q}).$$
\end{theorem}

This theorem provides a full characterization of when $\mathbb{Z}$-modules or ideals of the number field can provide a cyclic lattice, generalizing previous results from Fukshansky and Kogan \cite{Lenny}. We are now ready to proceed.

\section{Cyclic Lattices}
The definition of cyclic and quasi-cyclic lattices requires that the infinitely many points of the lattice remain in \(\Lambda\) after the rotation $\rho$ is applied. A finite algorithm to verify whether a lattice is quasi-cyclic consists of checking only the basis vectors of the lattice.

\begin{lemma} \label{teo2}
Let \( \Lambda \subset \mathbb{R}^n \) be a lattice with basis \( B = \{\vec{v}_1, \ldots, \vec{v}_m\} ,\) where \( m \leq n \). Then, \( \Lambda \) is {\(k\)-cyclic} (i.e., quasi-cyclic with cyclicity factor \( k \)) if and only if \( k \) is the smallest positive integer such that for every basis vector \( \vec{v}_i \in B \), the \( k \)-th rotation \( \rho^k(\vec{v}_i) \) belongs to \( \Lambda \).
\end{lemma}

\begin{proof}
The forward implication is trivial. For the converse, assume \( k \) is the smallest integer satisfying \( \rho^k(\vec{v}_i) \in \Lambda \) for all \( \vec{v}_i \in B \). Let \( \vec{x} \in \Lambda \) be an arbitrary lattice vector. Since \( B \) is a basis, \( \vec{x} \) can be written as an integer linear combination:
$\vec{x} = \sum_{i=1}^m a_i \vec{v}_i$, with $a_i \in \mathbb{Z}.$
Applying the \( k \)-th rotation \( \rho^k \) yields:
\(
\rho^k(\vec{x}) = \sum_{i=1}^m a_i \rho^k(\vec{v}_i).
\)
By hypothesis, \( \rho^k(\vec{v}_i) \in \Lambda \) for each \( i \), and since \( \Lambda \) is a lattice (closed under integer linear combinations), it follows that \( \rho^k(\vec{x}) \in \Lambda \). To show minimality, suppose for contradiction that there exists \( \ell < k \) such that \( \rho^\ell(\Lambda) = \Lambda \). Then, in particular, \( \rho^\ell(\vec{v}_i) \in \Lambda \) for all \( \vec{v}_i \in B \), contradicting the minimality of \( k \). Thus, \( \Lambda \) is \( k \)-cyclic.
\end{proof}

Since quasi-cyclic lattices are a generalization of cyclic lattices, it follows that the finite verification algorithm from the theorem above is valid particularly for the cyclic case, which proves the Theorem 1.1. Moreover, if we consider a k-cyclic lattice of rank n, the cyclicity factor k will always be a divisor of n:

\begin{proposition} 
Let $\Lambda$ be a $k$-cyclic lattice of dimension $n$. Then the cyclicity factor $k$ divides $n$.
\end{proposition}
\begin{proof}
By hypothesis, the lattice is $k$-cyclic. That is, $\rho^k(\Lambda)=\Lambda$. 
We have $n = kq + r$, where $k, q, r \in \mathbb{N}$ and $0 \leq r < k$. 
Then
$$\Lambda = \rho^n(\Lambda) = \rho^r(\rho^{kq}(\Lambda)) = \rho^r(\Lambda)  \Rightarrow \rho^r = \text{Id} \Rightarrow r = 0 \Rightarrow k | n.$$ \end{proof}

An important class of lattices is that of integral lattices, which are defined by the condition $\langle \vec{x},\vec{y} \rangle \in \mathbb{Z}$ for all $\vec{x},\vec{y} \in \Lambda$. In \cite{Lenny}, the cyclicity of certain families of integral lattices and their duals is investigated. Specifically, that paper shows that the $A_n$ lattices and their duals $A_n^*$ are both {strongly cyclic} lattices of rank $n$ in $\mathbb{R}^{n+1}$. Moreover, it proves that, for all $n \geq 2$, the $D_n$ lattices and their duals $D_n^*$ are both {cyclic} lattices of full rank in $\mathbb{R}^n$, being {strongly cyclic} if and only if $n$ is odd. Finally, \cite{Lenny} shows that the self-dual lattice $E_8 = E_8^*$ is {cyclic} but not {strongly cyclic} in $\mathbb{R}^8$, and that $E_6$ and $E_7$ are non-cyclic sublattices of $E_8$ in $\mathbb{R}^6$ and $\mathbb{R}^7$, respectively.

Among integral lattices, some subclasses are particularly noteworthy, such as even and odd lattices: an integral lattice is called {even} if $\langle \vec{x}, \vec{x} \rangle \in 2\mathbb{Z}$ for all $\vec{x} \in \Lambda$; otherwise, we say $\Lambda$ is {odd}. An important special case of integral lattices is that of unimodular lattices, which satisfy $\det(\Lambda) = \pm 1$. These lattices are self-dual ($\Lambda^* = \Lambda$). Particularly relevant are the \emph{even unimodular} lattices, which exist only in dimensions divisible by $8$ \cite{slone}. Notable examples include the $E_8$ lattice, the Barnes–Wall lattice $\mathrm{BW}_{16}$, and the Leech lattice $\Lambda_{24}$, which achieve the highest known sphere packing densities in their respective dimensions (8, 16, and 24).

As an application of Theorem 1.1 and Lemma 2.1, we verify the cyclicity of two of the lattices mentioned above. The conclusion is that neither the Barnes–Wall lattice $\mathrm{BW}_{16}$ nor the Leech lattice $\Lambda_{24}$ is cyclic or quasi-cyclic. To reach this conclusion, it suffices to check the condition for a single basis: if it fails for one basis, then it fails for the lattice. A suitable basis for this verification is given in \cite{sueli} (Table 2.2, p. 24 and Table 2.3, p. 25). Computing whether the rooted vectors of this basis lie in the lattice, we find that they do not; hence, neither lattice is cyclic or quasi-cyclic.

\section{Cyclic Algebraic Lattices}

We consider $\mathbb{K}$ as a number field of degree \( n \) and $\mathcal{O}_{\mathbb{K}}$ the  {ring of algebraic integers} of $\mathbb{K}$. It is a known fact that there exist $n$ distinct $\mathbb{Q}$-monomorphisms $\sigma_i: \mathbb{K} \rightarrow \mathbb{C}$, for $i=1,2,\ldots,n$. A monomorphism $\sigma_i$ is called  {real} if $\sigma_i(\mathbb{K})\subset \mathbb{R}$. Otherwise, it is called  {complex}. Let
$r$ be the number of real monomorphisms and 
$s$ the number of pairs of complex monomorphisms. It is a known fact that, $n = r + 2s$.

The trace of $\alpha \in \mathbb{K}$ is
$\operatorname{Tr}_{\mathbb{K}}(\alpha)
:= \sigma_1(\alpha) + \cdots + \sigma_n(\alpha)$.
The space
$K_{\mathbb{R}}=\mathbb{K}\otimes_{\mathbb{Q}}\mathbb{R}$
embeds into
$\mathbb{R}^{r}\times\mathbb{C}^{2s}\subseteq\mathbb{C}^n$ as
\[
\{(x,y)\in\mathbb{R}^{r}\times\mathbb{C}^{2s}
: y_{s+j}=\overline{y}_j,\ 1\leq j\leq s\}
\cong \mathbb{R}^{r}\times\mathbb{C}^{s},
\]
where $y=(y_1,\ldots,y_{2s})$. This space becomes a Euclidean
space via the bilinear form
$\langle\alpha,\beta\rangle
:=\operatorname{Tr}_{\mathbb{K}}(\alpha\overline{\beta})
\in\mathbb{R}$ for $\alpha,\beta\in\mathbb{K}$. For related results concerning trace forms over cyclic number fields, see \cite{OliveiraEtAl2017}. The Minkowski embedding $\sigma:\mathbb{K}\rightarrow K_{\mathbb{R}}\subset\mathbb{C}^n$,
given by
$\sigma(x)=(\sigma_1(x),\dots,\sigma_n(x))$,
sends any $\mathbb{Z}$-module $M$ of rank $n$ in $\mathbb{K}$,
including $\mathcal{O}_{\mathbb{K}}$ and its ideals, to a full-rank
lattice $\Lambda_M=\sigma(M)$ in $K_{\mathbb{R}}$. Such lattices
are called \textit{algebraic lattices}. Whenever $M=\mathcal{I}$
is a fractional or integral ideal of $\mathbb{K}$, the lattice
$\sigma(\mathcal{I})$ is called an \textit{ideal lattice}.  

As shown in \cite{Lenny}, $\sigma({\mathcal{O}_\mathbb{K}})$ is cyclic (for suitable ordering of embeddings) if and only if $\mathbb{K} / \mathbb{Q}$ is a cyclic extension. If $\theta$ denotes the generator of the Galois group of $\mathbb{K}$ over $\mathbb{Q}$, the specific embedding ordering stems from the cyclic nature of the shift operator $\rho$ and the cyclic Galois structure, given by:
\begin{equation} \label{eq1}
    \theta^{n-k+1} = \sigma_k, \quad 1 \leq k \leq n.
\end{equation}


In general, let $G=\textnormal{Gal}(\mathbb{K}/\mathbb{Q})$ be the Galois group of a number field $\mathbb{K}$ over $\mathbb{Q}$ and let $M$ be a $\mathbb{Z}$-submodule of $\mathcal{O}_\mathbb{K}$. For $\tau \in G$, we say that $M$ is \textit{$\tau$-stable} if $\tau(M)\subseteq M$. If $\tau(M)\subseteq M$ for all $\tau \in G$, we say that $M$ is \textit{$G$-stable}. Since every element of $G$ is bijective, we note that $M$ be $G$-stable is equivalent to $\tau(M)=M$ for all $\tau \in G$. With these definitions and notation, we are ready to prove Theorem 1.2.

\begin{proof}
Let $\alpha \in M$. As $\mathbb{K}$ is a number field of degree $n$, then $n=r+2s$, where $r$ is the number of real embeddings and $2s$ is the number of complex ones. Since the number field is Galois, it is a known fact that is either totally imaginary $(n = 2s)$ or totally real $(n = r)$. Furthermore, from the fact that $\mathbb{K}$ is cyclic, it follows that $Gal(\mathbb{K}/\mathbb{Q})=\langle \theta \rangle$. Suppose that $M$ is $G$-stable, that is, $\sigma_i(M) \subseteq M$ for all $i$. 
Consider the ordering given in (3.1). Observe that $\theta^n=id$. Then, for every element $\sigma(\alpha) \in \Lambda_M$, we have

\[\begin{array}{ll}
\rho(\sigma(\alpha)) &  = \rho(\sigma_1(\alpha),\sigma_2(\alpha), \ldots, \sigma_n(\alpha)) \\ & = \rho(\theta^{n-1+1}(\alpha), \theta^{n-2+1}(\alpha), \ldots, \theta^{n-n+1}(\alpha)) \\ & = \rho(\theta^{n}(\alpha), \theta^{n-1}(\alpha), \ldots, \theta(\alpha)) \\
&= \rho(\alpha, \theta^{n-1}(\alpha), \ldots, \theta(\alpha)) \\ & = (\theta(\alpha), \alpha, \theta^{n-1}(\alpha), \ldots, \theta^2(\alpha)).
\end{array}\]
On the other hand, 
\[\begin{array}{ll}
\sigma(\theta(\alpha)) & = (\sigma_1(\theta(\alpha)),\sigma_2(\theta(\alpha)), \ldots, \sigma_n(\theta(\alpha))) \\ & = (\theta^n(\theta(\alpha)), \theta^{n-1}(\theta(\alpha)), \ldots, \theta(\theta(\alpha)) ) \\ &= (\theta(\alpha), \alpha, \theta^{n-1}(\alpha), \ldots, \theta^2(\alpha)).
\end{array}\] That is, $\rho(\sigma(\alpha)) = \sigma(\theta(\alpha)).$ As $M$ is $G$-stable, $\theta(\alpha) \in M,$ for all $\alpha \in M$, it follows that $\sigma(\theta(\alpha)) \in \Lambda_M$. Therefore, $\rho(\sigma(\alpha)) \in \Lambda_M$, that is, $\Lambda_M$ is closed under $\rho$. Hence, $\Lambda_M$ is cyclic. Conversely, suppose $\Lambda_M$ is cyclic, for some $M \subset \mathcal{O}_\mathbb{K}$. Then, for every $\alpha \in M$, we have  $\rho^k(\sigma(\alpha)) \in \Lambda_M$, for $1 \leq k \leq n$. Note that
\( \rho(\sigma(\alpha)) = \rho(\sigma_1(\alpha),\sigma_2(\alpha), \ldots, \sigma_n(\alpha)) = (\sigma_n(\alpha),\sigma_1(\alpha), \ldots, \sigma_{n-1}(\alpha))\in \Lambda_M. \) Therefore, there exists $\beta \in M$ such that $\sigma(\beta) \in \Lambda_M$ satisfies 
\( (\sigma_n(\alpha),\sigma_1(\alpha), \ldots, \sigma_{n-1}(\alpha))= (\sigma_1(\beta),\sigma_2(\beta), \ldots, \sigma_{n}(\beta)) \).
Then $\sigma_n(\alpha) = \sigma_1(\beta)$, that is, $\sigma_n(\alpha) = \beta \in M$, and also $\sigma_i(\alpha) = \sigma_{i+1}(\beta)$, for $i=1, \ldots, n-1$. In an analogous way, if we apply the operator $\rho$ again, it follows that, for some $\gamma \in M$, $\sigma_{n-1}(\alpha) = \sigma_1(\gamma) = \gamma \in M$. Repeating this process, we obtain that $\sigma_i(\alpha) \in M$, for all $i$.
\end{proof}

\begin{example} A trivial example that satisfies the condition of Theorem 1.2 is $M = \mathcal{O}_\mathbb{K}$, which has already been studied in \cite{Lenny}.
\end{example}

\begin{example} Let $\mathbb{K}$ be a cyclic number field of prime degree $p>2$ with conductor $n$. Let $t = Tr_{\mathbb{Q}(\zeta_n)/\mathbb{K}}(\zeta_n)$ and $m$ a positive integer. Consider the family of modules of $\mathcal{O}_\mathbb{K}$ given by
\[ \mathcal{M}_m = \left\{
 \alpha = \sum_{i=0}^{p-1} a_i \theta^i(t) \in \mathcal{O}_\mathbb{K} :
 a_0, \ldots, a_{p-1} \in \mathbb{Z}, \
\sum_{i=1}^{p-1} a_i \equiv 0 \pmod{m}
\right\}.
\] The modules $\mathcal{M}_m$ were proposed in \cite{nunes2017} and are consolidated in the literature for lattice construction purposes \cite{L5,R24}. In \cite{L4}, it was proved that when $p$ is unramified in $\mathbb{K}$, the $\mathbb{Z}$-modules $\mathcal{M}_m$, for $m \in \mathbb{Z}^*_+$, generalize the ramified prime ideals in $\mathbb{K}$. Notice that $\sigma_i({\mathcal{M}_m}) = {\mathcal{M}_m}$, for all $i = 1, \ldots, n$, i.e., ${\mathcal{M}_m}$ is $G$-stable. Therefore, the lattice $\Lambda_{\mathcal{M}_m}=\sigma(\mathcal{M}_m)$ is cyclic. 
In \cite{R24}, it was proved that $\mathcal{M}_{m}$ has rank $p$ and has a $\mathbb{Z}$-basis given by $\{mt, \theta(t)-t, \theta^{2}(t)-t, \ldots, \theta^{p-1}(t)-t\}$; this allows us to conclude that $\mathcal{M}_m$ is also strongly cyclic. 
\end{example}

\begin{example} Let $f(x) \in \mathbb{Z}[x]$ be a monic irreducible polynomial of degree $n \geq 2$ with splitting field $\mathbb{K}$. Consider the family of modules \(\mathcal{M}_f = \langle \alpha_1, \alpha_2, \ldots, \alpha_n \rangle \subset \mathcal{O}_\mathbb{K} \) proposed in \cite{FukshanskyKnight2025}, where $\alpha_1, \alpha_2, \ldots, \alpha_n$ are the roots of $f(x)$. Let $\text{Gal}(\mathbb{K}/\mathbb{Q}) = \{\sigma_1, \ldots, \sigma_n\}$. Then for each $1 \leq i \leq n$ and $1 \leq j \leq n$, $\sigma_j(\alpha_i) = \alpha_k$ for some corresponding $1 \leq k \leq n$. Consider $\Lambda_{M_f} = \sigma(M_f)$ the Euclidean lattice in $K_\mathbb{R}$. Notice that the Galois action sends one root to another. Thus, taking any element of the module $\mathcal{M}_f$ and applying any element of the Galois group, we obtain another root, that is, another element in $\mathcal{M}_f$. This means that $\mathcal{M}_f$ is $G$-stable; hence, by Theorem 1.2, all lattices generated from this family of modules are cyclic. In \cite{FukshanskyKnight2025}, the authors present hypotheses on the coefficients of $f(x)$ that yield well-rounded lattices under the Minkowsky embedding.\end{example}

In the following, let $\mathbb{K}$ be a Galois number field of degree $n$ and let $G=\textnormal{Gal}(\mathbb{K}/\mathbb{Q})$ denote its Galois group. Let $M$ be a $\mathbb{Z}$-submodule of $\mathbb{K}$ not necessarily of rank $n$. By restricting the automorphisms in $G$ by $M$, we have a new set $$G_M := \{\tau|_M~:~\tau\in G\}.$$ 

Notice that $G_M$ is not a subgroup of $G$, since the restriction elements in $G_M$ are not the same of that in $G$. In general, $G_M$ is not a group itself.  The following lemma gives a necessary and sufficient condition for $G_M$ to be a group and provides other consequences: 

\begin{lemma}\label{lemma_G_M} Let $\mathbb{K}$ be a Galois number field, $G=\textnormal{Gal}(\mathbb{K}/\mathbb{Q})$ be its Galois group, and $M$ be a $\mathbb{Z}$-submodule of $\mathbb{K}$ of rank $k \leq n$. Then:
\begin{enumerate}
    \item $G_M$ is group if and only if $M$ is $G$-stable;
    \item If $G_M$ is group, then $|G_M|$ divides $|G|$;
    \item If $G_M$ is group and $G$ is cyclic, then $G_M$ is cyclic.
\end{enumerate}
\end{lemma}
\begin{proof} The first item is straightforward. Assuming that $G_M$ is a group, the other items follow from the fact that $G_M\cong  G/\textnormal{ker}(f)$, where $f:G\longrightarrow G_M$ is the group epimorphism defined as $f(\tau)=\tau|_M$ for all $\tau\in G$.
\end{proof}

\begin{corollary}
 Let $\mathbb{K}$ be a Galois number field of degree $n$ and let $G=\textnormal{Gal}(\mathbb{K}/\mathbb{Q})$ be its Galois group. If $G_M$ is a group and $M$ is a $\mathbb{Z}$-submodule of $\mathbb{K}$ of rank $n$, then $G_M$ is isomorphic to $G$.
\end{corollary}
\begin{proof} From Lemma 3.1, we have $G_M \cong  G/\textnormal{ker}(f)$, so it only remains to prove that $\textnormal{ker}(f) = \{ \tau \in G \;  | \; \tau|_M = \text{id}_\mathbb{K}\} = \{ \tau \in G \;  | \; \tau(m) = m, \; \forall m \in M\}$ is equal to
$\{\text{id}_\mathbb{K}\}$. Since $M$ has rank $n$, any basis of $M$ over $\mathbb{Z}$ is a basis of $\mathbb{K}$ over $\mathbb{Q}$. Every element $\tau$ of $G$ in the kernel fixes all elements of this basis, so $\tau = \text{id}_\mathbb{K}$. Thus, $\textnormal{ker}(f) = \{ \text{id}_\mathbb{K}\}$. Hence $G_M \cong  G$.
\end{proof}

This corollary is relevant because it shows that in the full-rank case, $G_M$ can be identified with the entire Galois group $G$. For $\mathbb{Z}$-modules of smaller rank, Lemma 3.1 can also be satisfied; that is, there exist $\mathbb{Z}$-modules of non-full rank that are $G$-stable, in which case $G_M$ is a group. An example can be constructed using the polynomial $f(x) = x^3 - 2$. Its roots are $\sqrt[3]{2}, \sqrt[3]{2} \omega, \sqrt[3]{2} \omega^2$, where $\omega = e^{2\pi i/3}$. The splitting field of is $\mathbb{K}=\mathbb{Q}(\sqrt[3]{2},\omega)$, which has degree 6 and Galois group $S_3$. If we consider the $\mathbb{Z}$-module $M=\mathbb{Z}[\omega]$ or the $\mathbb{Z}$-module $M=\langle\omega,\omega^2 \rangle_{\mathbb{Z}}$, the rank is 2 in a number field of degree 6. Notice that the map defined in the proof of Lemma 3.1 has $\text{ker}(f)=A_3$. Hence, $G_M \cong S_3/A_3$. 

Not every $\mathbb{Z}$-module of $\mathbb{K}$ is $G$-stable. An example of a non-$G$-stable submodule is a prime ideal $\mathcal{P} \in \mathcal{O}_\mathbb{K}$ above a prime $p \in \mathbb{Z}$. In this case, this prime ideal is only fixed by the elements of the decomposition group $\mathcal{D} = \{\sigma \in G \; | \; \sigma(\mathcal{P}) = \mathcal{P}\}$, which can be a proper subgroup of the Galois group $G$. An example where the decomposition group is a proper subgroup of $G$ is when the rational prime $p$ is unramified and decomposed. In this case, $|D| = e \cdot f$, but $|G| = e \cdot f \cdot g$, where $e$ is the ramification index, $f$ is the inertia degree (or residue field degree), and $g$ is the number of distinct prime ideals in $L$ lying above ${p}$. When $g > 1$, it follows that $|D| < |G|$, so $D \subsetneq G$. This allows us to conclude that $D$ and $G_M$ will only coincide if $\mathcal{P}$ is the only ideal in the ring of integers above the rational prime $p$, which proves the following corollary.

\begin{corollary} Let $\mathbb{K}$ be a Galois number field. If $\mathcal{P}$ is a prime ideal in $\mathcal{O}_\mathbb{K}$ above $p \in \mathbb{Z}$, then $G_M$ is a group if and only if $p$ totally splits or is inert in $\mathcal{O}_\mathbb{K}$.
\end{corollary}

With the results above, we have sufficient tools to prove Theorem 1.3:

\begin{proof} Consider $G=\{\sigma_1,\sigma_2,\ldots,\sigma_n\}$ with $\sigma_1=\textnormal{id}$. If $\Lambda$ is a cyclic lattice, for all $\alpha \in M$ there exists $\beta\in M$ such that $\rho(\sigma(\alpha))=\sigma(\beta)$, which implies
$$\beta=\sigma_n(\alpha) = \sigma_2^{-1}(\sigma_1(\alpha))=\sigma_3^{-1}(\sigma_2(\alpha))=\ldots = \sigma_n^{-1}(\sigma_{n-1}(\alpha)).$$
This means that $\sigma_n|_M=\sigma_k^{-1}|_M\circ \sigma_{k-1}|_M$ for $k=2,\ldots,n$. Thus, recursively, it is possible to see that $\sigma_k|_M = \sigma_2|_M^{k-1}$ for all $k=2,\ldots,n$. Therefore, $G_M = \langle \sigma_2|_M\rangle$, which is a cyclic group. Conversely, suppose that $G_M=\langle \tau|_M\rangle$ is cyclic group generated for some $\tau\in G$. It follows from Lemma 3.1 that $|G_M|$ divides $|G|=n$. Consequently, $\tau|_M^n=\textnormal{id}|_M$. The Minkowski embedding $\sigma$ restricted to $M$ can be defined, for all $\alpha\in M$, as $$\sigma|_M(\alpha)=(\textnormal{id}|_M(\alpha),\tau|_M(\alpha),\tau|_M^2(\alpha)\ldots,\tau|_M^{n-1}(\alpha)).$$ For each $\alpha\in M$, consider $\beta=\tau|_M^{n-1}(\alpha)\in M$. Since $\tau|_M^n=\textnormal{id}|_M$, then
$$\rho(\sigma(\alpha))=\rho(\sigma|_M(\alpha))=(\tau|_M^{n-1}(\alpha),\textnormal{id}|_M(\alpha),\ldots,\tau|_M^{n-2}(x)) = \sigma|_M(\beta)= \sigma(\beta).$$
This means that the lattice $\Lambda$ is cyclic.  
\end{proof}

By Lemma 3.1, $G_M$ is a group if and only if $M$ is $G$-stable. Thus, it is a consequence of the last theorem that $\Lambda=\sigma(M)$ is not a cyclic lattice if $M$ is not $G$-stable. However, clearly, it is not sufficient for $M$ to be G-stable for $\Lambda$ to be a cyclic lattice.

\begin{corollary} 
Let $n\in\mathbb{Z}$ a positive number. If $\mathbb{K}$ is a cyclic number field, then the lattice $\sigma(n\mathcal{O}_\mathbb{K})$ is cyclic.
\end{corollary} \begin{proof}
  Let $\theta$ be a generator of the Galois group $G$ of $\mathbb{K}$ over $\mathbb{Q}$. Then $\theta(n\mathcal{O}_\mathbb{K})=\theta(n)\theta(\mathcal{O}_\mathbb{K})=n\mathcal{O}_\mathbb{K}$, which means that $n\mathcal{O}_\mathbb{K}$ is $G$-stable. So $G_{n\mathcal{O}_\mathbb{K}}$ is a group by Lemma 3.1. It follows also from Lemma 3.1 that $G_{n\mathcal{O}_\mathbb{K}}$ is cyclic. Finally, Theorem 1.2 implies that $\sigma(n\mathcal{O}_\mathbb{K})$ is a cyclic lattice.  \end{proof}

\begin{remark}
Suppose that $\mathbb{K}$ is a cyclic number field and consider a $\mathbb{Z}$-module $M$ of $\mathbb{K}$ that is $G$-stable. By Lemma 3.1, $G_M$ is a cyclic group, once $G$ is cyclic. So Theorem 1.2 implies that $\Lambda=\sigma(M)$ is a cyclic lattice. This provides an alternative proof of one direction of Theorem 1.2.
\end{remark}
 
Let $\mathcal{P}$ and $\mathcal{Q}$ be prime ideals in $\mathcal{O}_\mathbb{K}$. We say that $\mathcal{P}$ and $\mathcal{Q}$ are \textit{conjugate} if there exists $\tau\in G=\textrm{Gal}(\mathbb{K}/\mathbb{Q})$ such that $\tau(\mathcal{P})=\mathcal{Q}$. It is a known fact in Algebraic Number Theory that two prime ideals $\mathcal{P}$ and $\mathcal{Q}$ in the ring of integers of a Galois number field are conjugate if and only if there are above the same prime number $p\in\mathbb{Z}$; moreover, the action of the Galois group of $\mathbb{K}$ over $\mathbb{Q}$ on the set of all prime ideals $S=\{\mathcal{P}_1,\mathcal{P}_2,\ldots,\mathcal{P}_k\}$ lying above the same prime number $p$ in $\mathbb{Z}$ permutes these prime ideals, that is, $S$ is $G$-stable \cite{neukirch} (chapter 1, $\S 9$).

Considering that the $G$-stability of a module $M$ determines if $G_M$ is a group (Lemma 3.1), we now characterize the ideals of a Galois number field $\mathbb{K}$ that are $G$-stable. Next, we denote by $\nu_\mathcal{I}(\mathcal{P})\geq 0$ the expoent of the prime ideal $\mathcal{P}$ in the factorization of a fractional or integral ideal $\mathcal{I}$ in $\mathbb{K}$.

\begin{lemma}
  \label{lemma_stable} Let $\mathbb{K}$ be a Galois number field and $\mathcal{I}$ be a (fractional or integral) ideal in $\mathbb{K}$. The ideal $\mathcal{I}$ is $G$-stable if and only if
$$\mathcal{P}~\textnormal{and}~\mathcal{Q}~\textnormal{are conjugate} ~\Longrightarrow~ \nu_\mathcal{I}(\mathcal{P})=\nu_\mathcal{I}(\mathcal{Q}).$$  
\end{lemma} 
\begin{proof} Consider the factorization in prime ideals of $\mathcal{O}_\mathbb{K}$ of $\mathcal{I} \in \mathcal{O}_\mathbb{K}$ given by $\mathcal{I}=\prod_{i=1}^n \prod_{j=1}^{m_i} \mathcal{P}_{ij}^{e_{ij}}$, where $e_{ij}\in\mathbb{Z}$ and $\mathcal{P}_{i1},\ldots,\mathcal{P}_{im_i}$ denote all the prime ideals above the same prime number $p_i\in\mathbb{Z}$ in $\mathbb{K}$, for each $i=1,\ldots,n$. For each $i=1,\ldots,n$, denote $\mathcal{I}_i=\prod_{j=1}^{m_i} \mathcal{P}_{ij}^{e_{ij}}$ - so $\mathcal{I}=\mathcal{I}_1\ldots \mathcal{I}_n$. On one side, we assume that $\mathcal{I}$ is $G$-stable. Let $\mathcal{P}=\mathcal{P}_{iu}$ and $\mathcal{Q}=\mathcal{P}_{iv}$ be any conjugate pair of prime ideals above the same prime $p_i$. Denote $e=e_{iu}$ and $f=e_{iv}$. Notice that $\nu_\mathcal{I}(\mathcal{P})= e$ and $\nu_\mathcal{I}(\mathcal{Q}) = f$. By definition, there exists $\tau\in G=\operatorname{Gal}(\mathbb{K}/\mathbb{Q})$ such that $\tau(\mathcal{P})=\mathcal{Q}$. Since $\tau$ only permutes the ideals above the same prime number, then $\tau(\mathcal{I}_i)=\mathcal{I}_i$. So, when we apply $\tau$ to $\mathcal{I}_i$, we have 
$$\mathcal{P}^{e}\mathcal{Q}^{f}\prod_{\substack{j=1\\j\not\in\{u,v\}}}^{m_i} \mathcal{P}_{ij}^{e_{ij}} = \mathcal{I}_i =\tau(\mathcal{I}_i) = \underbrace{\tau(\mathcal{P}^{e})}_{\mathcal{Q}^{e}}\tau(\mathcal{Q}^{f})\prod_{\substack{j=1\\j\not\in\{u,v\}}}^{m_i} \tau(\mathcal{P}_{ij}^{e_{ij}})$$
implying that $e=f$ due to the unicity of the decomposition of $\mathcal{I}_i$. This means $\nu_\mathcal{I}(\mathcal{P})=\nu_\mathcal{I}(\mathcal{Q})$. On the other hand, suppose that whenever $\mathcal{P}_{iu}$ and $\mathcal{P}_{iv}$ are conjugate for some $i=1,\ldots,n$, then $\nu_\mathcal{I}(\mathcal{P}_{iu})=\nu_\mathcal{I}(\mathcal{P}_{iv})$. As it is valid for any prime ideal in the factorization of $\mathcal{I}_i$, we have that all the exponents are equal, i.e., $\mathcal{I}_i=\mathcal{P}_{i1}^{e}\mathcal{P}_{i2}^{e}\ldots\mathcal{P}_{im_i}^{e}$. Applying any $\tau \in G$, we obtain $$\tau(\mathcal{I}_i) = \tau(\mathcal{P}_{i1}^{e}\mathcal{P}_{i2}^{e}\ldots\mathcal{P}_{im_i}^{e}) =  \tau(\mathcal{P}_{i1})^e\tau(\mathcal{P}_{i2})^{e}\ldots\tau(\mathcal{P}_{im_i})^{e}=\mathcal{I}_i$$ where the last equality follows because $S=\{\mathcal{P}_{i1},\mathcal{P}_{i2},\ldots,\mathcal{P}_{im_i}\}$ is $G$-stable. Thus, $\mathcal{I}_i$ is $G$-stable. Repeating this for all $i$ shows that $\mathcal{I}$ is $G$-stable. \end{proof}

The theorem above can be restated as follows: $\mathcal{I}$ is $G$-stable if and only if the prime ideals dividing $\mathcal{I}$ in $\mathbb{K}$ that are conjugate to one other occur with the same exponent in the factorization of $\mathcal{I}$ in $\mathbb{K}$. Theorem 1.4 is an immediate consequence of this result, yielding a characterization of the ideals in a cyclic number field that yield cyclic lattices under the Minkowski embedding.

\section*{Acknowledgements}
The authors thank Dr. Lenny Fukshansky and Dr. Trajano P. Nóbrega Neto for their valuable comments and suggestions. This work was supported by Conselho Nacional de Desenvolvimento Científico e Tecnológico (CNPq) [grant number 405842/2023-5] and Fundação de Amparo à Pesquisa do Estado de São Paulo (FAPESP) [grant numbers 2022/02303-0, 2025/01467-7]. This study was financed in part by the Coordenação de Aperfeiçoamento de Pessoal de Nível Superior - Brasil (CAPES) - Finance Code 001.

\end{document}